\documentclass[12pt]{amsart}
\usepackage{amsmath,amssymb,amsthm}
\usepackage{setspace}
\usepackage{moreverb}
\usepackage{mathrsfs}
\usepackage{url}
\usepackage[all]{xy}
\usepackage{lscape}
\usepackage{tikz}
\usetikzlibrary{arrows}
\usepackage{color}

\newcommand{\ignore}[1]{}

\newtheorem{theorem}{Theorem}[section]
\newtheorem{lemma}[theorem]{Lemma}
\newtheorem{corollary}[theorem]{Corollary}
\newtheorem{proposition}[theorem]{Proposition}

\theoremstyle{definition}
\newtheorem{definition}[theorem]{Definition}
\newtheorem{example}[theorem]{Example}

\theoremstyle{remark}
\newtheorem{remark}[theorem]{Remark}

\numberwithin{equation}{section}

\input xy
\xyoption{all}

\newcommand{\bC}{\mathbb{C}}

\newcommand{\bN}{{\mathbb{N}}}
\newcommand{\bZ}{{\mathbb{Z}}}

\newcommand{\bP}{{\mathbb{P}}}
\newcommand{\bQ}{{\mathbb{Q}}}
\newcommand{\bR}{\mathbb{R}}

\newcommand{\adj}{\mathrm{Adj}}

\newcommand{\fM}{{\mathfrak m}}

\newcommand{\fa}{\mathfrak{a}}

\newcommand{\cJ}{{\mathcal J}}

\newcommand{\cO}{{\mathcal O}}

\newcommand{\lra}{{\longrightarrow}}

\newcommand{\A}{\mathfrak{a}}

\newcommand{\Z}{\mathbb{Z}}
\newcommand{\Q}{\mathbb{Q}}

\newcommand{\Oc}{\mathcal{O}}
\newcommand{\m}{\mathfrak{m}}

\newcommand{\J}{\mathcal{J}}

\begin{document}

\title[Poincar\'e series of multiplier ideals]{Poincar\'e series of multiplier ideals in two-dimensional local rings with rational singularities}

\author[M. Alberich]{Maria Alberich-Carrami\~nana}
%\thanks{$^\dag$Partially supported by SGR2009-1284 and}

\author[J. \`Alvarez ]{Josep \`Alvarez Montaner}

\address{Deptartament de Matem\`atiques\\
Universitat Polit\`ecnica de Catalunya\\ Av. Diagonal 647, Barcelona
08028, Spain} \email{Maria.Alberich@upc.edu, Josep.Alvarez@upc.edu}

\author[F. Dachs] {Ferran Dachs-Cadefau }
\address{KU Leuven\\ Department of Mathematics \\ Celestijnenlaan 200B box 2400\\
BE-3001 Leuven, Belgium}

\email{Ferran.DachsCadefau@wis.kuleuven.be}

\author[V. Gonz\'alez]{V\'ictor Gonz\'alez-Alonso}
%\thanks{$^\dag$Partially supported by SGR2009-1284 and}
\address{Institut f\"ur Algebraische Geometrie\\ Leibniz Universit\"at Hannover \\
Welfengarten 1, 30167 Hannover, Germany} \email{gonzalez@math.uni-hannover.de}

\thanks{All four authors were partially supported by Generalitat de Catalunya SGR2014-634 project  and
Spanish Ministerio de Econom\'ia y Competitividad
MTM2015-69135-P.  FDC is also supported by the
KU Leuven grant OT/11/069. VGA is also supported by the ERC StG
279723 ``Arithmetic of algebraic surfaces'' (SURFARI).
 MAC is also with the  Institut de Rob\`otica i Inform\`atica Industrial (CSIC-UPC)
 and the Barcelona Graduate School of Mathematics (BGSMath).}

%\keywords {Multiplier ideals, jumping numbers, plane curves.}

%\subjclass[2000]{Primary 13D45, 13N10}

\begin{abstract}
We study the multiplicity of the jumping numbers of an $\fM$-primary
ideal $\fa$ in a two-dimensional local ring with a rational singularity.
The formula we provide for the multiplicities leads to a very simple
and efficient method to detect whether a given rational number
is a jumping number. We also give an explicit description of the Poincar\'e series
of multiplier ideals associated to $\fa$ proving, in particular, that it is a rational function.
\end{abstract}

\maketitle

\section{Introduction}

Let $X$ be a complex surface with a rational singularity at a point $O\in X$
and $\cO_{X,O}$ its corresponding local ring.  Let $\fa \subseteq \cO_{X,O}$ be an $\fM$-primary
ideal where $\fM=\fM_{X,O}$ is the maximal ideal of $\cO_{X,O}$. Then, for any real exponent $c>0$,
we may consider its corresponding \emph{multiplier ideal} $\J(\fa^c)$.
It turns out that the multiplier ideal becomes smaller as the parameter $c$ grows and, whenever we
have an strict inclusion $\J(\fa^{c -\varepsilon}) \varsupsetneq \J(\fa^c)$ for arbitrarily small
$\varepsilon > 0$, we say that $c$ is a jumping number.

\vskip 2mm

Since $\fa $ is $\fM$-primary, its associated
multiplier ideals are $\fM$-primary as well so they have finite codimension, as $\bC$-vector spaces, in $\cO_{X,O}$. This fact
prompted Ein-Lazarsfeld-Smith-Varolin \cite{ELSV04} to define the \emph{multiplicity} of a jumping number as
the codimension as $\bC$-vector spaces of two consecutive multiplier ideals. In general, for any positive real number $c$
we can define its multiplicity as
$$m(c):= \dim_{\bC} \frac{\J(\A^{c-\varepsilon})}{ \J(\A^{c})}$$ where $\varepsilon$ is small enough. In particular, $c$ is
a jumping number whenever $m(c)>0$.
In order to gather all the information given by all jumping numbers and their corresponding multiplicities,
Galindo-Monserrat \cite{GM10} introduced the so-called  {\it Poincar\'e series of multiplier ideals} associated to $\fa$
as the series with fractional exponents $$P_\A (t)= \sum_{c\in \bR_{>0}} m(c) \hskip 1mm t^{c}.$$ The main result
in \cite{GM10} is the fact that the Poincar\'e series of a simple complete $\fM$-primary ideal $\fa \subseteq
\cO_{X,O}$, for a smooth point $O$,
is rational in the sense that it belongs to the field of fractional functions $\bC(z)$ where the indeterminate $z$
corresponds to a fractional power $t^{1/e}$ for a suitable $e\in \bN_{>0}$.
They also provided a closed formula for $P_\A (t)$ that relies in J\"arviletho's formula \cite{Jar11} for the set
of jumping numbers.

\vskip 2mm

One of the goals of this paper is to extend their result to the case of any  $\fM$-primary ideal in
a surface with a rational singularity at $O$. To do so we provide first a systematic study of the
multiplicities  using the theory of {\it jumping divisors} introduced in \cite{ACAMDC13}. Another goal that we
achieve is to give a simple numerical criterion (see Theorem \ref{cond_suf}) which characterizes whether any given rational
number is a jumping number.

\vskip 2mm

The paper is organized as follows: First we briefly recall the basics on the theory of multiplier ideals
and the aspects on the theory of singularities that we will use throughout this work.

\vskip 2mm

In Section \S3 we review the notion of {\it jumping divisors}
introduced in \cite{ACAMDC13}. In fact we will be mainly interested in
the maximal jumping divisor since it satisfies a nice periodicity property.
In particular we will give a geometrical description of this divisor.
We also point out that, {\em en passant}, we provide several technical
results that will be crucial in the rest of the paper.

\vskip 2mm

The core of the paper can be found in Section \S4. We provide two different formulas to describe
the multiplicity  for any $c\in \bR_{>0}$. The first one (see
Theorem \ref{multiplicity1_c}) is described in terms of the maximal jumping divisor associated
to $c$. The periodicity of this divisor leads to Proposition \ref{growth} that provides a very
clean description of the growth of multiplicities in terms of dicritical components of the maximal
jumping divisor. This is the key result that we will use in the description of the Poincar\'e series
associated to $\fa$ in the final section. The second formula for the multiplicity (see Proposition \ref{virtual2}) is given
using the notion of {\it virtual codimension} introduced in \cite{Cas00} and \cite{Reg96}.

\vskip 2mm

In Section \S5 we provide  a very simple (and efficient) algorithm to compute the set of jumping numbers of $\fa$.
It boils down to compute the multiplicities of the rational numbers in the set of {\it candidate jumping numbers}.
This relies on a simple numerical criterion to characterize jumping numbers (see Theorem \ref{cond_suf}).
Another consequence of the formulas for the multiplicities is that we can describe those jumping numbers
contributed by dicritical divisors. In particular we give in Theorem \ref{jn1-2} a full description of the
jumping numbers in the interval $(1,2]$.

\vskip 2mm

The main result of Section \S6 is a description of the Poincar\'e series of multiplier ideals
for any $\fM$-primary ideal $\fa$. As a consequence, we can easily recover the case of simple ideals obtained
by Galindo-Monserrat \cite{GM10}  in the smooth case.
Finally we relate the Poincar\'e series to the Hodge spectrum of a generic element $f\in \fa$.
In particular we recover an old result of L{\^e} V{\u{a}}n Th{\`a}nh-Steenbrink \cite{LVTS89}
describing the Hodge spectrum of a plane curve.

%%%%%%%%%%%%%%%%%%%%%%%%%%%%%%%%%%%%%%%%%%%%%%%%%%%%%%%%%%%%%%%%%%%%%%%%%%%%%%%%%%%%%%%%%%%%%%%%%%

\section{Preliminaries}
Let $\left(X,O\right)$ be a germ of complex surface with at worst a rational singularity.
Let $\cO_{X,O}$ denote the local ring at $O$, $\fM = \fM_{X,O} \subseteq \cO_{X,O}$ the maximal ideal, and
let $\fa \subseteq \fM$ be an $\fM$-primary ideal. Recall that a {\em log-resolution} of the pair
$\left(X,\fa\right)$ (or of $\fa$, for short) is a birational morphism $\pi: X' \rightarrow X$ such that
\begin{enumerate}
\item $X'$ is smooth (in particular, $\pi$ is a resolution of the singularity),
\item the exceptional locus $E = Exc\left(\pi\right)$ is a divisor with simple normal crossings (the irreducible components $E_1,\ldots,E_r$ of $E$ are all smooth and intersect transversely), and
\item the preimage of $\fa$ is locally principal, that is, $\fa\cdot\cO_{X'} = \cO_{X'}\left(-F\right)$ for some effective divisor $F$ supported on $E$.
\end{enumerate}

The theory of rational singularities was introduced by Artin in \cite{Art66} and further
developed by Lipman in \cite{Lip69}. We recall that the point $O$ being (at worst) a rational singularity means that $R^1\pi_*\cO_{X'}=0$ for some (hence any) desingularization. A first consequence of Artin's results is that the exceptional divisor of any desingularization is a tree of rational curves. Indeed, according to \cite[Proposition 1]{Art66} a singularity is rational if and only if any effective divisor $D$ with exceptional support has arithmetic genus (see \cite[Page 486]{Art62})
$$p_a\left(D\right) = 1 + \frac{1}{2}\left(K_{X'}+D\right)\cdot D \leqslant 0.$$
Since the components $E_i$ of the exceptional divisor are smooth, we have $p_a\left(E_i\right) \geq 0$, hence $p_a\left(E_i\right)=0$, which means that they are rational. Furthermore, there cannot be a {\em cycle} $E_1,\ldots,E_k$ of exceptional components (i.e., such that $E_1 \cdot E_2 = E_2 \cdot E_3 = \cdots = E_1 \cdot E_k = 1$ and $E_i \cdot E_j=0$ for any other $i \neq j$), since the formula $p_a\left(A+B\right) = p_a\left(A\right) + p_a\left(B\right) + A\cdot B - 1$ would give $p_a\left(E_1+\cdots+E_k\right) = 1$.

\vskip 2mm

The above numerical characterization \cite[Proposition 1]{Art66} of rational singularities is not satisfying enough, since it involves testing every effective exceptional divisor. In the same work, Artin proved in \cite[Theorem 3]{Art66} that it is enough to check the {\it fundamental cycle}, the
unique smallest non-zero effective divisor $Z$
(with exceptional support) such that
$$Z\cdot E_i \leqslant 0 \qquad \text{for every } \, i=1,\dots, r.$$
Another important property of the fundamental cycle is that
$\fM\cdot\cO_{X'} = \cO_{X'}\left(-Z\right)$, hence any desingularization is a log-resolution of the maximal ideal $\fM$.
.

\vskip 2mm

Since rational
singularities are $\bQ$-factorial, it is possible to define a {\em relative canonical divisor} $K_{\pi}$ of $\pi$,
which can be characterized as the unique divisor $K_{\pi} = \sum_{i=1}^r k_i E_i$ supported on the
exceptional divisor and such that
\begin{equation} \label{eq-relative-canonical}
\left(K_{\pi}+E_j\right)\cdot E_j = \left(\sum_{i=1}^r k_i E_i \cdot E_j\right) + E_j^2 = -2
\end{equation}
for every exceptional component $E_j$ (because of the adjunction formula). Note that the coefficients $k_i$ are uniquely
determined because the intersection matrix $\left(E_i \cdot E_j\right)_{i,j}$ is negative-definite, but they are not necessarily
integral nor positive. Moreover, due to this numerical characterization, $K_{X'}$ can be replaced by $K_{\pi}$ to compute the arithmetic genus as $p_a(Z)=1+ \frac{1}{2} \left(K_\pi + Z\right)\cdot Z$.

\vskip 2mm

\vskip 2mm

The ideal $\fa$ being $\fM$-primary, $F$ is supported on the exceptional locus, hence it can be written as
$F= \sum_{i=1}^r e_i E_i$ for some positive integers $e_i$. For any component $E_i$, the {\em excess} of $\fa$ at $E_i$ is defined as
\begin{equation} \label{eq-excess}
\rho_i = - F \cdot E_i \geqslant 0.
\end{equation}
If $C$ is a curve through $O$ defined by a general element in $\fa$, then $\rho_i$ is the number of branches of the strict transform $\widetilde{C}$ that intersect $E_i$. The {\em total excess} is defined as $\rho = \sum_{i=1}^r \rho_i,$ and is therefore the number of
branches at $O$ of a general curve of the linear system defined by $\fa$. In particular, $\rho > 0$.

\vskip 2mm

For any $\bR$-divisor $D = \sum_i d_i D_i$ in $X'$, where the $D_i$ are pairwise different prime divisors,
its {\em round-down} $\left\lfloor D \right\rfloor$, {\em round-up} $\left\lceil D \right\rceil$ and
{\em fractional part} $\left\{D\right\}$ are defined by applying the corresponding operation to the coefficients $d_i$.

\vskip 2mm

The {\em multiplier ideal (sheaf)} associated to $\fa$ and some real number $c \in \bR$ is defined as
$$\J\left(\fa^c\right) = \pi_*\cO_{X'}\left(\left\lceil K_{\pi} - cF \right\rceil\right).$$
Since $\fa$ is $\fM$-primary, any multiplier ideal $\J\left(\fa^c\right)$ is also $\fM$-primary. Furthermore,
for any $\varepsilon > 0$ it holds $\J\left(\fa^c\right) \supseteq \J\left(\fa^{c+\varepsilon}\right)$, with equality for
$\varepsilon$ small enough. Hence the multiplier ideals form a discrete nested sequence
$$\Oc_{X,O}\varsupsetneq\J(\A^{\lambda_1})\varsupsetneq\J(\A^{\lambda_2})\varsupsetneq...\varsupsetneq\J(\A^{\lambda_i})\varsupsetneq...$$
indexed by an increasing sequence of rational numbers $0 < \lambda_1 < \lambda_2 < \ldots$ such that
$\J(\A^{\lambda_i})=\J(\A^c)\varsupsetneq\J(\A^{\lambda_{i+1}})$ for any $c \in \left[\lambda_i,\lambda_{i+1}\right)$. The
$\lambda_i$ are the so-called {\em jumping numbers} of the ideal $\fa$. We point out now
two properties that will be useful in the sequel:
\begin{itemize}
\item (local vanishing) for any $c \in \bR$, it holds $R^1\pi_*\cO_{X'}\left(\left\lceil K_{\pi} - cF \right\rceil\right)=0$, and
\item (Skoda's theorem) $\J\left(\fa^c\right)=\fa \J\left(\fa^{c-1}\right)$ for any $c > 2$.
\end{itemize}
For further properties and some applications of
multiplier ideals, we refer the reader to the book of Lazarsfeld \cite{Laz04}.

Being $\fM$-primary, the multiplier ideals have finite $\bC$-codimension in $\cO_{X,O}$. This fact prompted Ein,
Lazarsfeld and Varolin \cite{ELSV04} to define the {\em multiplicity} of $\lambda_i$ as
$$m\left(\lambda_i\right) = \dim_{\bC}\frac{\J\left(\fa^{\lambda_{i-1}}\right)}{\J\left(\fa^{\lambda_i}\right)}.$$
Since $\J\left(\fa^{\lambda_{i-1}}\right) = \J\left(\fa^{\lambda_i-\varepsilon}\right)$ for small $\varepsilon$, we can
extend this definition to any $c \in \bR$ as
\begin{equation} \label{eq-def-mult-general}
m\left(c\right) := \dim_{\bC}\frac{\J\left(\fa^{c-\varepsilon}\right)}{\J\left(\fa^c\right)}
\end{equation}
With this definition, it is clear that $c$ is a jumping number if and only if $m\left(c\right)>0$.

\vskip 2mm

In order to describe the behavior of the jumping numbers and its multiplicities, Galindo and Montserrat
\cite{GM10} introduced the {\em Poincar\'e series of multiplier ideals} associated to $\fa$, which after our definition of
multiplicity can be written as
\begin{equation} \label{eq-def-Poinc}
P_\A (t)= \sum_{c\in \bR_{>0}} m(c) \hskip 1mm t^{c}.
\end{equation}

We introduce now some technical notation. Given any exceptional component $E_i$, define
$$\adj\left(E_i\right)=\left\{E_j \hskip 2mm | \hskip 2mm  E_i \cdot E_j = 1\right\} \quad \text{and} \quad a\left(E_i\right) = \#\adj\left(E_i\right) = E_i \cdot \left(E-E_i\right),$$
the set of exceptional components adjacent to $E_i$ and its number. More generally, for any reduced exceptional divisor
$D=E_{i_1}+\cdots+E_{i_m}$ define
$$\adj_D\left(E_i\right) = \left\{E_j \leqslant D \hskip 2mm | \hskip 2mm E_i \cdot E_j = 1\right\} \quad \text{and} \quad a_D\left(E_i\right) = \#\adj_D\left(E_i\right),$$
the set of components adjacent to $E_i$ {\em inside} $D$. Define also the set of components adjacent to $D$ as
$$\adj\left(D\right)=\left\{E_j \hskip 2mm | \hskip 2mm E_j \not\leqslant D \, \text{and} \, D \cdot E_j = 1\right\}.$$
Finally, denote by $v_D = m$ (resp. $a_D$) the number of irreducible components of $D$ (resp. intersections between two components of $D$). Since
the exceptional set is a tree of rational curves, any $D$ as before is a collection of trees of rational curves, and it is then clear that
$$\sum_{E_i \leqslant D}a_D\left(E_i\right) = 2 a_D$$
and that $v_D - a_D$ equals the number of connected components of $D$. We also say that $E_i$ is an {\em end} of
$D$ if $E_i \leqslant D$ and $a_D\left(E_i\right)=1$.

\vskip 2mm

Finally we mention that there are two kinds of exceptional divisors that will play a special role throughout this work:

\begin{enumerate}
\item[$\bullet$] An exceptional component $E_i$ is a {\em rupture} component if $a\left(E_i\right) \geqslant 3$, that is, it intersects at least three more components of $E$ (different from $E_i$).
\item[$\bullet$] We say that $E_i$ is {\em dicritical} if $\rho_i > 0$. Dicritical components correspond to
{\it Rees valuations} by \cite{Lip69}.
\end{enumerate}

\section{Jumping divisors}

Recall from \cite[Definition 4.1]{ACAMDC13} that a {\em jumping divisor} for a jumping number $\lambda$ is a reduced
exceptional divisor $G$ such that $\lambda e_i - k_i \in \bZ$ for every irreducible component $E_i \leqslant G$, and for
small $\varepsilon > 0$ satisfies
\begin{equation} \label{jumping-property}
\cJ\left(\fa^{\lambda-\varepsilon}\right) = \pi_*\cO_{X'}\left(\left\lceil K_{\pi}-\lambda F\right\rceil+G\right).
\end{equation}
That is, $G$ gives a jump from the multiplier ideal with exponent $\lambda$ to the previous one. In \cite{ACAMDC13} it was
proved that, given a jumping number $\lambda$, every jumping divisor $G$ satisfies $G_{\lambda} \leqslant G \leqslant H_{\lambda}$
for some special jumping divisors $G_{\lambda}$ and $H_{\lambda}$. These divisors are called respectively {\em minimal} and
{\em maximal} jumping divisor, and the former is extensively studied in \cite{ACAMDC13}. The aim of this section is to study
the maximal one, which can be defined for any positive real number $c$ and will play a prominent role in the rest of the paper.

\begin{definition} \label{df-maximal-jumping}
Given any real number $c \in \bR$, we define its associated {\em maximal jumping divisor} as
\begin{equation} \label{eq-def-Hc}
H_c = \left\lceil K_{\pi}-\left(c-\varepsilon\right) F\right\rceil -\left\lceil K_{\pi}-cF\right\rceil
\end{equation}
for a sufficiently small $\varepsilon>0$. Alternatively, it can be defined as the reduced divisor whose components are the exceptional curves $E_i$ such that $k_i-ce_i \in \bZ$.
\end{definition}

It follows immediately from the definition that the maximal jumping divisors satisfy the following periodicity property.

\begin{lemma} \label{periodic_jd}
For any real number $c \in \bR$, we have $H_c=H_{c+1}$.
\end{lemma}

\begin{remark}
The definition of minimal jumping divisors given in \cite[Definition 4.3]{ACAMDC13} is more involved and
is closely related to the algorithm given in loc. cit. for the computation of the chain of multiplier ideals.
Is for this reason that minimal jumping divisors are only defined for jumping numbers in \cite{ACAMDC13}.
 However one may extend the definition to any positive real number $c$ if we consider $G_c=0$
 for any non-jumping number $c > 0$.  Notice that the equality (\ref{jumping-property}) is still trivially
 satisfied for any divisor $G$ such that $G_c \leqslant G \leqslant H_c$.
Regarding the periodicity of the minimal jumping divisor, we only  have $G_c=G_{c+1}$ for
$c>1$ (see \cite[Proposition 4.8]{ACAMDC13}) and there are examples where this equality does not hold for $c \leqslant 1$.
\end{remark}

We focus now on the structure of $H_c$. We first prove some formulas to compute its intersection with its irreducible and connected components.

\begin{lemma} \label{num}
Fix $c \in \bR_{\geqslant 0}$ and consider a component $E_i$ of the jumping divisor $H_c$. Then
$$\left(\left\lceil K_\pi-c F\right\rceil + H_c\right)\cdot E_i = -2 + c \rho_i + a_{H_c}\left(E_i\right) + \sum_{ E_j \in \adj(E_i)} \left\{c e_j-k_j\right\}.$$
\end{lemma}
\begin{proof}
For any $E_i\leqslant H_c$  we have
\begin{multline*}
(\left\lceil K_{\pi}-c F\right\rceil + H_c)\cdot E_i = (\left(K_\pi-c F\right)+\left\{-K_\pi+c F\right\} + H_c - E_i + E_i)\cdot E_i = \\
= \left(K_{\pi}+ E_i\right)\cdot E_i - cF\cdot E_i + \left(H_c-E_i\right)\cdot E_i + \left\{cF-K_{\pi}\right\}\cdot E_i.
\end{multline*}
Let us now compute each summand separately. The first three terms are easy: $\left(K_{\pi}+E_i\right)\cdot E_i = -2$ follows from the adjunction formula, $-cF \cdot E_i = c\rho_i$ holds by definition, and clearly $a_{H_c}\left(E_i\right) = \left(H_c-E_i\right)\cdot E_i$ because $E_i \leqslant H_c$. It only remains to prove that
\begin{equation} \label{eq-decimal-part}
\left\{cF-K_{\pi}\right\}\cdot E_i = \sum_{ E_j \in \adj(E_i)}
\left\{c e_j-k_j\right\},
\end{equation}
which is also quite immediate. Indeed, writing $\left\{cF-K_{\pi}\right\} = \sum_{j=1}^r\left\{ce_i-k_i\right\}E_j$, (\ref{eq-decimal-part}) follows by observing that, for $j \neq i$, $E_j \cdot E_i = 1$ if and only if $E_j \in \adj\left(E_i\right)$, and the term corresponding to $j=i$ vanishes because we assumed $E_i \leqslant H_c$, hence $ce_i-k_i \in \bZ$.
\end{proof}

%{\bf He afegit aquest corol·lari, que \'es \'util a la demo del Teorema \ref{thm-ends-Hc}}

\begin{corollary} \label{cor-integer}
For any $c \in \bR_{>0}$ and any $E_i \leqslant H_c$, the sum
$$c \rho_i + \sum_{ E_j \in \adj(E_i)} \left\{c e_j-k_j\right\}$$
is an integer.
\end{corollary}

\begin{proposition} \label{Num_condition_c}
Fix any $c\in \bR_{>0}$, and let $H_c$ be its associated maximal jumping divisor. Then the following inequalities hold:
\begin{itemize}
\item $\left(\left\lceil K_{\pi}- cF\right\rceil + H_c\right)\cdot E_i \geqslant -1$ for all $E_i\leqslant H_c$, and
\item $\left(\left\lceil K_{\pi}- cF\right\rceil + H_c\right)\cdot H \geqslant -1$ for any connected component $H\leqslant H_c$.
\end{itemize}
\end{proposition}
\begin{proof}
From Lemma \ref{num} we already know that $\left(\left\lceil K_\pi-cF\right\rceil + H_c\right)\cdot E_i  \geqslant -2$ for all $E_i\leqslant H_c$. If equality holds, then it must also hold
\begin{itemize}
\item $a_{H_c}\left(E_i\right)=0$, that is, $E_i$ is an isolated component in $H_c$,
\item $\left\{ce_j - k_j\right\} = 0$ for all $E_j \in \adj\left(E_i\right)$, that is, every exceptional component $E_j$ intersecting $E_i$ is also contained in $H_c$, and
\item $\rho_i=0$.
\end{itemize}
The first two conditions imply that $E_i$ is the only exceptional curve of the log-resolution. But in this case $\rho_i = \rho > 0$ and the third condition is not satisfied.

\vskip 2mm

As for the second part, using Lemma \ref{num} for all $E_i \leqslant H$ and summing up we obtain
\begin{align*}
\left(\left\lceil K_\pi- cF\right\rceil + H_c\right)\cdot H  & = -2v_H +\sum_{E_i \leqslant H} \left( \sum_{ E_j \in \adj(E_i)} \left\{ce_j - k_j\right\} + c\rho_i \right) + 2a_H \\
& = -2 + \sum_{E_i \leqslant H} \left( \sum_{ E_j \in \adj(E_i)}\left\{ce_j - k_j\right\}\right) + c \sum_{E_i \leqslant H} \rho_i \geqslant -2,
\end{align*}
where $a_H-v_H=1$ due to the tree structure of the exceptional divisor and the connectedness of $H$. Equality holds if and only if
$$\sum_{E_i \leqslant H} \sum_{ E_j\in \adj(E_i)} \left\{ce_j - k_j\right\}=\sum_{E_i \leqslant H} c\rho_i=0.$$
The first condition implies that $H$ is the whole exceptional divisor, and then the second condition implies that $\rho = 0$, which is impossible. Hence the inequality must be strict, and since $\left(\left\lceil K_\pi-cF\right\rceil + H_c\right)\cdot H\in\bZ$, the claim follows.
\end{proof}

We will now get some insight on the topology of the $H_c$.

\begin{theorem} \label{thm-ends-Hc}
Fix any $c \in \bR_{>0}$, and let $H_c$ be the corresponding maximal jumping divisor. Then:
\begin{itemize}
\item The isolated components of $H_c$ must be either a rupture divisor, a dicritical divisor or a divisor $E_i$ with $a\left(E_i\right) = 2$ such that $$\sum_{ E_j \in \adj(E_i)} \left\{c e_j -k_j\right\} = 1.$$
\item An end of a reducible connected component of $H_c$ must be either a rupture divisor, a dicritical divisor or an end of the whole exceptional divisor.
\end{itemize}
\end{theorem}
\begin{proof}
Let $E_i$ be an isolated component of $H_c$. Assume that it is neither a rupture nor a dicritical component.
Then it only has one or two adjacent components in the exceptional divisor. In the first case, if $E_j$ is the only exceptional
component in $\adj\left(E_i\right)$, then the formula given in Lemma \ref{num} reduces to
$\left(\left\lceil K_\pi-c F\right\rceil + H_c\right)\cdot E_i= -2 + \left\{ c e_j - k_j\right\}$.
Since $\left\{ c e_j - k_j\right\}<1$, we would get $\left(\left\lceil K_\pi-c F\right\rceil + H_c\right)\cdot E_i < -1$,
contradicting Proposition \ref{Num_condition_c}. The only possible remaining case is $a\left(E_i\right)=2$.
If $\adj\left(E_i\right)=\left\{E_j,E_l\right\}$, then we have
$\left(\left\lceil K_\pi-c F\right\rceil + H_c\right)\cdot E_i= -2 + \left\{ c e_j - k_j\right\} + \left\{c e_l - k_l\right\}$. Since
$$0\leqslant\left\{ c e_j - k_j\right\} + \left\{c e_l - k_l\right\} < 2$$
must be an integer by Corollary \ref{cor-integer} (we assumed $E_i$ to be non-dicritical, i.e. $\rho_i=0$), it must equal 0 or 1.
But the former contradicts Proposition \ref{Num_condition_c}, hence the only possibility is that
$\left\{ c e_j - k_j\right\} + \left\{c e_l - k_l\right\} = 1$, which is the last possibility given in the statement.

\vskip 2mm

As for the second assertion, let $E_i$ be an end of a reducible connected component of $H_c$ that is neither a rupture divisor,
nor a dicritical divisor nor an end of the whole exceptional divisor. Then it has two adjacent components in the whole exceptional
divisor, say  $E_j$ and $E_l$, but only one of them, say $E_j$, is in $H_c$. Then we have
$$\left(\left\lceil K_\pi-c F\right\rceil + H_c\right)\cdot E_i= -2 + \left\{c e_l -k_l\right\} +1 \not \in \bZ,$$
which is impossible.
\end{proof}

There are examples where any of these cases is achieved, in particular we may find isolated components of $H_c$ that are
neither a rupture nor a dicritical divisor.

\begin{example}
Consider the ideal $\A=(x^3,y^{10}) \subseteq \bC\{x,y\}$. Its minimal log-resolution
has six exceptional components $E_1,\ldots,E_6$ indexed according to the order in which they are obtained by successive blow-ups.
They are arranged as the following dual graph shows

\vskip 2mm
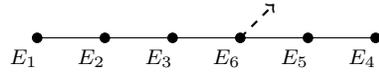
\begin{center}
\begin{tikzpicture}[scale=0.9]
  \draw  (0,0) -- (5,0);
  \draw [dashed,->,thick] (3,0) -- (3.5,0.5);
  \draw (-.2,-0.3) node {{\tiny $E_1$}};
  \draw (.8,-0.3) node {{\tiny $E_2$}};
  \draw (1.8,-0.3) node {{\tiny $E_3$}};
  \draw (4.8,-0.3) node {{\tiny $E_4$}};
  \draw (3.8,-0.3) node {{\tiny $E_5$}};
  \draw (2.8,-0.3) node {{\tiny $E_6$}};
  \filldraw  (0,0) circle (2pt)
             (1,0) circle (2pt)
             (2,0) circle (2pt)
             (3,0) circle (2pt)
             (4,0) circle (2pt)
             (5,0) circle (2pt);
\end{tikzpicture}
\end{center}

\vskip 2mm
\noindent where the dashed arrow indicates that $E_6$ is the only dicritical component, with excess $\rho_6=1$.
The relative canonical divisor is $K_\pi= E_1+2E_2+3E_3+4E_4+8E_5+12E_6$ and the divisor $F$ such that
$\fa\cdot\cO_{X'} = \cO_{X'}\left(-F\right)$ is $F= 3E_1+6E_2+9E_3+10E_4+20E_5+30E_6$.

\vskip 2mm The maximal jumping divisor associated to $c=\frac{3}{2}$ is $H_{\frac{3}{2}}= E_2+E_4+E_5+E_6$.
It has two connected components, one of which ($E_2$) is as predicted at the first statement of Theorem \ref{thm-ends-Hc}.

\end{example}

%%%%%%%%%%%%%%%%%%%%%%%%%%%%%%%%%%%%%%%%%%%%%%%%%%%%%%%%%%%%%%%%%%%%%%%%%%%%%%%%%

\section{Multiplicities of Jumping Numbers}

Let $\fa \subseteq \cO_{X,O}$ be an $\fM$-primary ideal. The aim of this section
is to describe the multiplicity
$$m(c)= \dim_{\bC} \frac{\J(\A^{c-\varepsilon})}{ \J(\A^{c})}$$
for any real exponent $c>0$, where $\varepsilon$ is small enough.
In Theorem \ref{multiplicity1_c} we will give a formula
described in terms of the maximal jumping divisor associated
to $c$. This formula and  Proposition \ref{growth} will be  the key ingredients for
 the description of the Poincar\'e series
associated to $\fa$  that we will give in Theorem \ref{serie_Poincare}.

\vskip 2mm

We will also provide a second formula for the multiplicity in Proposition \ref{virtual2} that is based
on the concept of {\it virtual codimension} considered by Casas-Alvero \cite{Cas00}
and Reguera \cite{Reg96} for the smooth  and  the rational singularities case respectively.

\vskip 2mm

We start with the first formula.

\begin{theorem}\label{multiplicity1_c}
Let $\fa \subseteq \cO_{X,O}$ be an $\fM$-primary ideal and $H_c$
the maximal jumping divisor associated to some $c \in \bR_{>0}$. Then,
\begin{align*}
m\left(c\right) & = \left(\left\lceil K_{\pi}-cF\right\rceil +
H_c\right)\cdot H_c +\#\left\{\text{connected components of }
H_c\right\}.
\end{align*}
\end{theorem}
\begin{proof}
Consider the short exact sequence
$$0 \longrightarrow \Oc_{X'}\left(\left\lceil K_\pi-cF\right\rceil\right) \longrightarrow \Oc_{X'}\left(\left\lceil K_\pi-cF\right\rceil + H_c\right) \longrightarrow \Oc_{H_c}\left(\left\lceil K_\pi-cF\right\rceil+H_c\right) \longrightarrow 0$$
Pushing it forward to $X$ and applying local vanishing for multiplier ideals we get the
short exact sequence
\begin{multline*}
0 \longrightarrow \pi_*\Oc_{X'}\left(\left\lceil K_\pi-cF\right\rceil\right) \longrightarrow \pi_*\Oc_{X'}\left(\left\lceil K_\pi-cF\right\rceil+H_c\right) \longrightarrow \\
\longrightarrow H^0\left(H_c,\Oc_{H_c}\left(\left\lceil
K_\pi-cF\right\rceil+H_c\right)\right) \otimes \bC_O \longrightarrow
0
\end{multline*}
or equivalently, since $H_c = \left\lceil
K_\pi-\left(c-\varepsilon\right)F\right\rceil-\left\lceil
K_\pi-cF\right\rceil$ for $\varepsilon$ small
enough,
$$0 \longrightarrow\J(\A^c) \longrightarrow\J(\A^{\left(c-\varepsilon\right)}) \longrightarrow H^0\left(H_c,\Oc_{H_c}\left(\left\lceil K_\pi-cF\right\rceil+H_c\right)\right) \otimes \bC_O \longrightarrow 0$$
Therefore the multiplicity of $c$ is just
\begin{align*}
m\left(c\right) & = h^0\left(H_c,\Oc_{H_c}\left(\left\lceil K_\pi-cF\right\rceil+H_c\right)\right) \\
& = \sum_{E_i \leqslant H_c} h^0\left(E_i,\Oc_{E_i}\left(\left\lceil
K_\pi-cF\right\rceil+H_c\right)\right) -  a_{H_c},
\end{align*}
where in the second equality we have used that $H_c$ has simple
normal crossings, and hence the sections of the line bundle
$\Oc_{H_c}\left(\left\lceil K_\pi-cF\right\rceil+H_c\right)$
correspond to sections over each component that agree on the
$a_{H_c}$ intersections. Indeed, we can consider the twist by
$\mathcal{O}_{X'}\left(\left\lceil K_\pi-cF\right\rceil+H_c\right)$
of the following exact sequence
$$0 \longrightarrow \mathcal{O}_{H_c} \longrightarrow \bigoplus_{E_i \leqslant H_c} \mathcal{O}_{E_i} \longrightarrow \bigoplus_{E_i,E_j \leqslant H_c} \mathcal{O}_{E_i \cap E_j} \longrightarrow 0,$$
where the summands in the last term are length-one skyscraper sheaves (due to the simple normal crossings condition), the first map is the direct sum of the restrictions $\mathcal{O}_{H_c} \rightarrow \mathcal{O}_{E_i}$ and the second map is given by the differences at the intersections $E_i \cap E_j$.

Recall now that each exceptional component $E_i$ is isomorphic to
$\bP^1$, and that the sections of a line bundle on $\bP^1$ are
determined by its degree (namely,
$h^0\left(\Oc_{\bP^1}\left(d\right)\right)= d+1$ if $d\geqslant -1$ and
zero otherwise). Then, using that
$$\deg \Oc_{E_i}\left(\left\lceil K_\pi-cF\right\rceil+H_c\right) = \left(\left\lceil K_\pi-cF\right\rceil+H_c\right)\cdot E_i \geqslant -1$$
by Proposition \ref{Num_condition_c}, we get
\begin{align*}
m\left(c\right) & = \sum_{E_i \leqslant H_c}\left(\left(\left\lceil K_\pi-cF\right\rceil+H_c\right)\cdot E_i + 1\right) -  a_{H_c} \\
& = \left(\left\lceil K_\pi-cF\right\rceil+H_c\right)\cdot H_c + v_{H_c} -  a_{H_c} \\
& = \left(\left\lceil K_\pi-cF\right\rceil+H_c\right)\cdot H_c +
\#\left\{\text{connected components of } H_c\right\}.
\end{align*}
\end{proof}

\begin{remark}
When $c=\lambda$ is a jumping number, the same formula for the
multiplicity can be described using the associated minimal jumping
divisor $G_\lambda$. Namely,
\begin{align*}
m(\lambda) & = \left(\left\lceil K_\pi-\lambda F\right\rceil +
G_{\lambda}\right)\cdot G_{\lambda} +\#\{\text{connected components
of } G_{\lambda}\}
\end{align*}
The proof of this result holds verbatim to the one given for Theorem \ref{multiplicity1_c}
but we have to refer to \cite[Proposition 4.16]{ACAMDC13} instead of Proposition \ref{Num_condition_c}.

\vskip 2mm

For reduced divisors in the interval $G_{\lambda}< G < H_{\lambda}$
we may have $E_i\leqslant G$ such that
$$\left(\left\lceil K_\pi-\lambda F\right\rceil+G\right)\cdot E_i=-2 + \sum_{ E_j \in \adj(E_i)} \left\{\lambda e_j-k_j\right\}  + \lambda \rho_i +a_{G}(E_i)= -2.$$
Namely, this happens when $E_i$ is a non-dicritical isolated
component of $G$ with all adjacent divisors in $H_\lambda$. However,
these divisors can also provide a formula for the multiplicity of a
jumping number as follows. Refining the arguments used in the proof
of Theorem \ref{multiplicity1_c} we obtain:
$$m(\lambda) = \left(\left\lceil K_\pi-\lambda F\right\rceil+G\right)\cdot G +\#\{\text{c.c. of } G\}  +\#\left\{ E_i \hskip 2mm | \hskip 2mm \left(\left\lceil K_\pi-\lambda F\right\rceil+G\right)\cdot E_i=-2\right\}.$$
\end{remark}

In some cases it will be more convenient to use the following
reinterpretation of the formula given in Theorem
\ref{multiplicity1_c}.

\begin{corollary} \label{multiplicity2}
Let $\fa \subseteq \cO_{X,O}$ be an $\fM$-primary ideal and $H_c$
the maximal jumping divisor associated to some $c \in \bR_{>0}$. Then,

\begin{align*}
m\left(c\right)  & = \sum_{E_i \leqslant H_c} \left( \sum_{ E_j \in
\adj(E_i)} \left\{c e_j -
k_j\right\}+c\rho_i\right)-\#\left\{\text{connected components of }
H_{c}\right\}.
\end{align*}

\end{corollary}
\begin{proof}
Using Lemma \ref{num} we have:
\begin{align*}
m(c) & = \left(\left\lceil K_\pi-c F\right\rceil+H_{c}\right)\cdot H_c + \#\left\{\text{connected components of } H_c\right\} \\
 & = \sum_{E_i \leqslant H_c} \left(-2 + \sum_{E_j \in \adj(E_i)} \left\{c e_j - k_j\right\} + c \rho_i + a_{H_c}\left(E_i\right)\right) + \#\left\{\text{c.c. of } H_c\right\} \\
 & = -2v_H + \sum_{E_i \leqslant H_c} \left(\sum_{E_j \in \adj(E_i)} \left\{c e_j - k_j\right\} + c \rho_i \right) + 2a_{H_c}  + \#\left\{\text{c.c. of } H_c\right\} \\
 & = \sum_{E_i \leqslant H_c} \left(\sum_{E_j \in \adj(E_i)} \left\{c e_j - k_j\right\} + c \rho_i \right) - \#\left\{\text{c.c. of } H_{c}\right\} \\
\end{align*}
\end{proof}

As an immediate consequence of this we obtain the following
slight generalization of a result of Tucker \cite[Proposition 7.3]{Tuc10}.
We point out that J\"arviletho already proved in \cite{Jar11} that
$1$ is not a jumping number for simple $\fM$-primary ideals.

\begin{corollary}
Suppose that $O$ is a smooth point, and let $\fa \subseteq \cO_{X,O}$ be an $\fM$-primary ideal. The
multiplicity of $c=1$ is $$m(1)= \rho -1.$$ In particular, $c=1$ is
a jumping number if and only if $\fa$ is not simple.
\end{corollary}
\begin{proof}
The maximal jumping divisor for $c=1$ has the same support as $F$, so the result
follows from Corollary \ref{multiplicity2}.
\end{proof}

From the formula given above and the periodicity of the maximal jumping
divisor $H_{c}$, it is easy to control the growth of the
multiplicities in terms of the excesses at dicritical components.
This result is a key point in the proof of Theorem
\ref{serie_Poincare}.

\begin{proposition} \label{growth}
Let $\fa \subseteq \cO_{X,O}$ be an $\fM$-primary ideal and $H_c$
the maximal jumping divisor associated to some $c \in \bR_{>0}$. Then,
$$m\left(c+1\right)- m\left(c\right) = \sum_{E_i \leqslant H_c} \rho_i.$$
In particular, $0 \leqslant m\left(c+1\right) - m\left(c\right) \leqslant
\rho$.
\end{proposition}
\begin{proof}
Recall that $c$ and $c + 1$ have the same jumping divisor $H_{c}$
(see Lemma \ref{periodic_jd}). Therefore, by Theorem
\ref{multiplicity1_c}, we have
$$m\left(c+1\right)- m\left(c\right)= -F\cdot H_{c}= \sum_{E_i \leqslant H_c} \rho_i.$$
\end{proof}

%%%%%%%%%%%%%%%%%%%%%%%%%%%%%%%%%%%%%%%%%%%%%%%%%

\subsection{Virtual codimensions}
Given an effective $\bR$-divisor $D=\sum d_i E_i$ with exceptional support  we may consider its
associated ideal (sheaf) $\pi_{\ast}\Oc_{X'}(-D):= \pi_{\ast}\Oc_{X'}(-\lceil D \rceil)$. Its stalk at $O$ is
an $\fM$-primary complete ideal of $\Oc_{X,O}$ that we will simply denote as $I_D$.
We say that  two divisors are {\it equivalent} if they  define the same ideal.  In the equivalence class of a given divisor $D$ one
may find
a unique maximal representative, its so-called {\it antinef closure}   $\widetilde{D}$ (see \cite[\S18]{Lip69}).
First, recall that an effective divisor with integer coeficients ${D'}$   is called
\emph{antinef} if $-{D'}\cdot E_i \geqslant 0$,  for every exceptional prime divisor
$E_i$.

 \vskip 2mm
 The antinef closure of $D$ can be computed using an inductive procedure
 called {\it unloading} that was already described
in the work of Enriques \cite[IV.II.17]{EC15} (see also \cite{Lau72}, \cite[\S 4.6]{Cas00} and \cite{Reg96}).
Here we will consider the version given by the first three authors in \cite{ACAMDC13}. Unloading values to any $D$ is
to consider the new divisor $$D'= \lceil D \rceil + \sum_{E_i \in
\Theta} n_i E_i,$$ where $\Theta$ is the set of components $E_i
\leqslant D$ with negative excesses, i.e.
$$\Theta:= \{E_i \leqslant D \hskip 2mm | \hskip 2mm  \rho_i=
-\lceil D \rceil \cdot E_i <0 \}$$ and $n_i= \left \lceil \frac
{\rho_i}{E_i^2} \right \rceil$. We say that the unloading is {\it
tame} if $\rho_i=-1$ for all $E_i\in \Theta$ and there are no
adjacent divisors in $\Theta$. This is a mild generalization of the
notion of tameness introduced in \cite{Cas00}. The antinef closure $\widetilde{D}$ of  $D$ is achieved
after finitely many unloading steps.

 \vskip 2mm

Given a  divisor  $D$ with exceptional support, we will define its {\it virtual codimension} or {\it virtual number of conditions} as

%\begin{definition}
%For any given $\Q$-Cartier divisor $D$, its virtual codimension is
$$\mathcal{C}(D):=  -\frac{\lceil D\rceil \cdot(\lceil D\rceil+K_\pi)}{2}.$$
%\end{definition}

The main feature of this invariant is that it coincides with the codimension of the associated ideal when $D$ is antinef.
For a proof of this result one may consult \cite[Proposition 4.7.1]{Cas00} for the smooth case  and   \cite[Proposition 3.7]{Reg96}
for the rational singularities case.

\begin{proposition} \label{virtual codim is codim}
 Let $D $ be an antinef divisor and $I_{D}$ its associated ideal. Then:
 $$\mathcal{C}(D)= \dim_{\bC} \cO_{X,O}/I_{D}$$
\end{proposition}

This result is no longer true for arbitrary divisors. However, there
are some non-antinef divisors for which this equality holds.

\begin{proposition} \label{tame_step}
Assume that a divisor $D' $ is obtained from a divisor $D $ by performing a single
unloading step. Then $ \mathcal{C}(D) \geqslant \mathcal{C}(D')$ and the equality holds if and only if the unloading step is tame.

\end{proposition}

\begin{proof}
Notice that, in order to compute the virtual codimension, we may
always assume $D=\lceil D \rceil$. Hence,  $D'=D +\sum_{E_i \in
\Theta} n_i E_i$, where $\Theta$  and $n_i= \left\lceil
\frac{\rho_i}{E_i^2}\right\rceil$
 are defined as above. Therefore:
 $$\begin{array}{rcl}
    \mathcal{C}(D)-\mathcal{C}(D')    &=  &-\frac{1}{2}\left({D}^2-{D'}^2+K_\pi\cdot(D-D')\right)\\
            &=  &-\frac{1}{2}\left(-2\left(\sum_in_iE_i\right){D}-\left(\sum_in_iE_i\right)^2-K_\pi\cdot\left(\sum_in_iE_i\right)\right)\\
            &=  &-\frac{1}{2}\left(-2\left(\sum_in_iE_i\right){D}-\left(\sum_in_iE_i\right)^2+2\sum_in_i+\sum_in_iE_i^2\right)\\
            &=  &\sum_i\frac{n_i}{2}\left(-2\rho_i+(n_i-1)E_i^2-2\right)+\sum_i\sum_{j>i}n_in_jE_i\cdot E_j\\
   \end{array}
 $$
We are assuming $n_i\geqslant 1$ for all $E_i \in \Theta$ so the summands
$\frac{n_i}{2}\left(-2\rho_i+(n_i-1)E_i^2-2\right)$ are always $\geqslant
0$. Notice that they are zero if and only if $\rho_i=-1$ for all
$E_i\in \Theta$. On the other hand, $\sum_i\sum_{j>i}n_in_jE_i\cdot
E_j \geqslant 0$ and equality holds if and only if $E_i\cdot E_j=0$ for all $E_i \neq E_j \in \Theta$, i.e.
there are no adjacent divisors in the set $\Theta$.
\end{proof}

\begin{corollary} \label{tame_unloading}
 Let $\widetilde{D}$ be the antinef closure of a divisor $D $ and $I_{D}$ their associated ideal, then:
 $$\mathcal{C}(D)\geqslant \mathcal{C}(\widetilde{D})= \dim_{\bC} \cO_{X,O}/I_{D}$$
 and the equality holds if and only if all the unloading steps performed to obtain $\widetilde{D}$
 are tame.
 \end{corollary}

When we deal with multiplier ideals we can extract a very simple formula for the multiplicity of any real number.

\begin{proposition}\label{virtual1}
Let $D_c$ and $D_{c-\varepsilon}$ be the antinef closures of $\left\lfloor c F - K_{\pi} \right\rfloor$
and $\left\lfloor (c-\varepsilon) F - K_{\pi} \right\rfloor$ respectively, for any $c\in \bR_{\geqslant 0}$ and $\varepsilon$ small enough.
Then, the multiplicity of $c$ is
$$m(c)=\mathcal{C}(D_c)-\mathcal{C}(D_{c-\varepsilon})= \frac{ D_{c-\varepsilon} \cdot( D_{c-\varepsilon}+K_\pi)}{2}
-\frac{ D_{c} \cdot( D_{c}+K_\pi)}{2}.$$
\end{proposition}

\begin{proof}

We have $$m(c)=\dim_{\bC} \cO_{X,O}/ \J\left(\fa^c\right) - \dim_{\bC} \cO_{X,O}/ \J\left(\fa^{c-\varepsilon}\right) $$
and, using Proposition \ref{virtual codim is codim}, the virtual codimensions coincide with the
codimension for antinef divisors so $m(c)= \mathcal{C}(D_c)-\mathcal{C}(D_{c-\varepsilon})$ and the result follows.

\end{proof}

Actually there is no need to compute the antinef closure of the aforementioned divisors to obtain the same result.

\begin{proposition} \label{virtual2}
For any $c\in \bR_{\geqslant 0}$ and $\varepsilon$ small enough we have
$$m(c)=\mathcal{C}(\left\lfloor c F - K_{\pi} \right\rfloor)-\mathcal{C}(\left\lfloor (c-\varepsilon) F - K_{\pi} \right\rfloor)=$$ $$=
\frac{\left\lfloor (c-\varepsilon) F - K_{\pi} \right\rfloor \cdot( \left\lfloor (c-\varepsilon) F - K_{\pi} \right\rfloor +K_\pi)}{2}
-\frac{ \left\lfloor c F - K_{\pi} \right\rfloor \cdot( \left\lfloor c F - K_{\pi} \right\rfloor +K_\pi)}{2}.$$
\end{proposition}

\begin{proof}
 Recall that $\lceil K_\pi-(c -\varepsilon) F\rceil=\lceil K_\pi-c
F\rceil+H_c $. Then:

\vskip 4mm

$ \mathcal{C}(\left\lfloor c  F - K_{\pi}
\right\rfloor)-\mathcal{C}(\left\lfloor c  F - K_{\pi} \right\rfloor - H_c )  = $

\vskip 2mm

$ \hskip 1cm = \frac{1}{2} (\left\lfloor c  F - K_{\pi}
\right\rfloor - H_c )\cdot(\left\lfloor c  F - K_{\pi} \right\rfloor -
H_c  +K_\pi) -   \frac{1}{2} (\left\lfloor c  F - K_{\pi} \right\rfloor
)\cdot(\left\lfloor c  F - K_{\pi} \right\rfloor +K_\pi)$

%\vskip 2mm

%$ = \frac{1}{2}(\left\lfloor c F - K_{\pi} \right\rfloor - H_c)\cdot(\left\lfloor c F - K_{\pi} \right\rfloor -
%H_c +K_\pi) -\frac{1}{2} (\left\lfloor c F - K_{\pi} \right\rfloor )\cdot(\left\lfloor c F - K_{\pi} \right\rfloor +K_\pi) $

\vskip 2mm

$ \hskip 1cm =   - \left\lfloor c  F - K_{\pi}
\right\rfloor\cdot H_c    +   \frac{H_c \cdot H_c}{2} - \frac{K_\pi \cdot H_c}{2}$

\vskip 2mm

$\hskip 1cm = %\frac{(2 \lceil K_\pi - c F \rceil + H_c  - K_\pi)\cdot H_c }{2} =
(\lceil K_\pi - c F \rceil + H_c )\cdot H_c - \frac{(H_c  + K_\pi)\cdot H_c }{2}$

\vskip 2mm

$\hskip 1cm = (\lceil K_\pi - c F \rceil + H_c )\cdot H_c + \#\{\text{connected components of }H_c\} = m(c).$

\vskip 4mm
Here we used the fact that
$$\frac{1}{2}(K_\pi+H_c) \cdot H_c=-v_{H_c}+a_{H_c}=-\#\{\text{connected components of }H_c\}$$ and
Theorem \ref{multiplicity1_c}.

\end{proof}

Let $  {\lambda}' < \lambda$ be two consecutive jumping numbers of an $\fM$-primary ideal $\fa \subseteq \cO_{X,O}$.
Despite the fact that $ \left\lfloor \lambda' F - K_{\pi} \right\rfloor$ and
$ \left\lfloor (\lambda-\varepsilon) F - K_{\pi} \right\rfloor$ have the same antinef closure their virtual codimensions
may differ. However, we still have the following description of the multiplicity

\begin{proposition} \label{virtual3}
Let $  {\lambda}' < \lambda$ be two consecutive jumping numbers of an $\fM$-primary ideal $\fa \subseteq \cO_{X,O}$.
Then, the multiplicity of $\lambda$ is
$$m(\lambda)=\mathcal{C}(\left\lfloor \lambda F - K_{\pi} \right\rfloor)-\mathcal{C}(\left\lfloor \lambda' F - K_{\pi} \right\rfloor)=$$ $$=
\frac{\left\lfloor \lambda' F - K_{\pi} \right\rfloor \cdot( \left\lfloor \lambda' F - K_{\pi} \right\rfloor +K_\pi)}{2}
-\frac{ \left\lfloor \lambda F - K_{\pi} \right\rfloor \cdot( \left\lfloor \lambda F - K_{\pi} \right\rfloor +K_\pi)}{2}.$$
\end{proposition}

\begin{proof}

Consider all the rational numbers $\gamma \in (\lambda',\lambda)$ for which there exists at
least one component $E_i$ such that $\gamma e_i -k_i\in\Z$. We order them  to form
a finite sequence of rational numbers $\lambda' < \gamma_1 < \cdots < \gamma_r < \lambda$.
Notice that these are the only rational numbers in this interval where the virtual codimension
of $\left\lfloor \gamma F - K_{\pi} \right\rfloor$ may increase. %differs from $\left\lfloor \lambda F - K_{\pi} \right\rfloor$

\vskip 2mm

We have $$m(\lambda)= \mathcal{C}(\left\lfloor \lambda F - K_{\pi} \right\rfloor)-\mathcal{C}(\left\lfloor (\lambda-\varepsilon) F - K_{\pi} \right\rfloor)=
\mathcal{C}(\left\lfloor \lambda F - K_{\pi} \right\rfloor)-\mathcal{C}(\left\lfloor \gamma_r F - K_{\pi} \right\rfloor)$$ and, at every step of the sequence,
$m(\gamma_i)= \mathcal{C}(\left\lfloor \gamma_i F - K_{\pi} \right\rfloor)-\mathcal{C}(\left\lfloor \gamma_{i-1} F - K_{\pi} \right\rfloor)$. Therefore
$$m(\lambda)=m(\lambda)+\sum_{i>0}m(\gamma_i)=\mathcal{C}\left(\lfloor\lambda F - K_\pi \rfloor\right)-\mathcal{C}\left(\lfloor\lambda' F - K_\pi\rfloor\right)\,$$
due to the fact that $m(\gamma_i)=0$ as these rational numbers are not jumping numbers.
\end{proof}

\vskip 2mm

\begin{remark}
In the case that $X$ is smooth we can check that the unloading steps needed to compute the antinef closure of
$\left\lfloor c F - K_{\pi} \right\rfloor$  for any $c\in \bR_{\geqslant 0}$ are tame. Indeed, repeating the same
arguments considered in the proof of Proposition \ref{virtual3} we may end up with the case $c=0$.
It is then easy to check that $\mathcal{C}\left(\lfloor -K_\pi\rfloor\right)= \mathcal{C}(D_{0})=0$ so we get
$$\mathcal{C}\left(\lfloor c F -K_\pi\rfloor\right)=\mathcal{C}(D_{c})\,.$$
This concludes the remark thanks to Corollary \ref{tame_unloading}.
\end{remark}

\section{Jumping Numbers via multiplicities}

Fix a log-resolution $\pi: X' \lra X$ of an $\fM$-primary ideal  $\fa \subseteq \cO_{X,O}$.
Consider the relative canonical divisor $K_\pi = \sum_{i=1}^r k_i E_i $,
and the divisor $F=\sum_{i=1}^r e_i E_i $ such that $\fa \cdot \cO_{X'}= \cO_{X'}(-F)$.
The jumps between multiplier ideals must occur at rational numbers that belong to
the set of {\it candidate jumping numbers}
$$\left \{  \frac{k_i + m}{e_i} \hskip 2mm | \hskip 2mm m\in \bZ_{>0}  \right\}.$$
Not every candidate jumping number is necessarily a jumping number. Using the formulas for the
multiplicity given in the previous section we can easily extract the set of jumping numbers
since we have:

\begin{proposition}\label{JN_m}
Let $\fa \subseteq \cO_{X,O}$ be an $\fM$-primary ideal and $c \in
\bR_{>0}$. Then, $c$ is a jumping number if and only if $m(c) > 0$.
\end{proposition}

In addition, we have the following simple criterion

\begin{theorem} \label{cond_suf}
Let $\fa \subseteq \cO_{X,O}$ be an $\fM$-primary ideal and $c \in
\bR_{>0}$. Then, there exists a connected  component $H\leqslant
H_c$ such that
$$\left(\left\lceil K_\pi-c F\right\rceil + H_c\right)\cdot H > -1$$ if and only if $m(c)>0$.
\end{theorem}

\begin{proof}

Theorem \ref{multiplicity1_c} states that the multiplicity of $c$ is:
\begin{align*}
m\left(c\right) & = \left(\left\lceil K_{\pi}-cF\right\rceil +
H_c\right)\cdot H_c +\#\left\{\text{connected components of }
H_c\right\} \\  & = \sum_{H\leqslant
H_c} \left( \left(\left\lceil K_\pi-c F\right\rceil + H_c\right)\cdot H +1\right),
\end{align*}
where the sum is taken over all the connected components $H\leqslant
H_c$. Then, the result follows since $\left(\left\lceil K_\pi-c F\right\rceil + H_c\right)\cdot H \geqslant -1$
by Proposition \ref{Num_condition_c}.
\end{proof}

\vskip 2mm Therefore we have a simple algorithm to compute the set of jumping numbers of $\fa$
that boils down to compute the multiplicity of the rational numbers in the set of
candidate jumping numbers by means of the formula given in Theorem \ref{multiplicity1_c}
or the one given in Proposition \ref{virtual2}. We have implemented this algorithm  in the Computer Algebra system
{\tt Macaulay 2} \cite{GS}.
The scripts of the source codes  as well as the output in full detail of some
examples will be available at the web page
\begin{center}
  {\tt  www.pagines.ma1.upc.edu/$\sim$jalvz/multiplier.html}
\end{center}
It turns out that this algorithm is more efficient than the algorithms considered by Tucker in \cite{Tuc10} and
the first three authors in \cite{ACAMDC13}.

\subsection{Jumping numbers contributed by dicritical divisors}
Another interesting consequence of the methods developed in the previous sections
is the fact that we can describe a big chunk of the set of jumping numbers
by means of an inspection of dicritical divisors. In the sequel we will consider
 a dicritical divisor $E_i$ with excess $\rho_i=-F\cdot E_i>0$ and value
$v_i(F)=e_i$.

\begin{theorem}
Let $\fa \subseteq \cO_{X,O}$ be an $\fM$-primary ideal. Let $k\in
\bN$ be a  non-negative integer such that
$\frac{k}{e_i}>\frac{1}{\rho_i}$. Then, $\lambda=\frac{k}{e_i}$ is a
jumping number.
\end{theorem}

\begin{proof}
Let $H\leqslant H_\lambda$ be the connected component that contains
the dicritical divisor $E_i$.  For
$\lambda=\frac{k}{e_i}>\frac{1}{\rho_i}$ we have

  \begin{eqnarray*}
( \left\lceil K_\pi- \lambda F\right\rceil + H_\lambda )\cdot H & =
& \sum_{ E_j \in \adj(H)} \left\{ \lambda e_j -k_j\right\} +
\sum_{E_j \leqslant H_\lambda} \lambda \rho_j  -2   \\
&> & \sum_{ E_j \in \adj(H)} \left\{ \lambda e_j -k_j\right\} +
\sum_{\stackrel{E_j \leqslant H_\lambda}{j\neq i}} \lambda \rho_j +1
-2 \hskip 2mm \geqslant \hskip 2mm -1
\end{eqnarray*}
and the result follows from Theorem \ref{cond_suf}.
\end{proof}

For the boundary case $\lambda = \frac{1}{\rho_i}$ we have the
following criteria.

\begin{proposition}
Let $\fa \subseteq \cO_{X,O}$ be an $\fM$-primary ideal. Let $k\in
\bN$ be a  non-negative integer such that $\frac{k}{e_i} = \frac{1}{\rho_i}$. Then, the following are equivalent:
 \begin{itemize}
  \item[i)] $\lambda = \frac{1}{\rho_i}$ is not a jumping number.
  \item[ii)] $H_\lambda=E$ is the whole exceptional component, and $E_i$ is the only dicritical divisor.
\end{itemize}
\end{proposition}

\begin{proof}
Let $H\leqslant H_\lambda$ be the connected component that contains
the dicritical divisor $E_i$.  For
$\lambda=\frac{k}{e_i}=\frac{1}{\rho_i}$ we have
$$( \left\lceil K_\pi- \lambda F\right\rceil + H_\lambda )\cdot H  =  \sum_{ E_j \in \adj(H)} \left\{ \lambda e_j -k_j\right\} +
\sum_{\stackrel{E_j \leqslant H_\lambda}{j\neq i}} \lambda \rho_j +1
-2$$ By Theorem \ref{cond_suf},
$\lambda=\frac{k}{e_i}=\frac{1}{\rho_i}$ is not a jumping number
 when this intersection multiplicity is $-1$.
Notice that a divisor $E_j$ satisfies $\left\{ \lambda e_j
-k_j\right\}=0$ if and only if $E_j \leqslant H_\lambda$. Thus
$$\sum_{ E_j \in \adj(H)} \left\{ \lambda e_j -k_j\right\}=0$$ if
and only if $\adj(H)=\emptyset$, or equivalently when $H_\lambda = E$. On the other hand
$$\sum_{\stackrel{E_j \leqslant H_\lambda}{j\neq i}} \lambda
\rho_j=0$$ if and only if $\rho_j=0$ for all $j\neq i$, i.e. when
there are no dicritical divisors besides $E_i$.

\end{proof}

Notice that the result above also generalizes the fact that $1$ is not a jumping
number for simple $\fM$-primary ideals. We can also extend to our setting
J\"arviletho's result on the behavior of the jumping numbers in the
interval $(1,2]$ given in \cite[Theorem 9.9]{Jar11} for simple complete ideals in a smooth surface.

\begin{theorem} \label{jn1-2}
 Let $\fa \subseteq \cO_{X,O}$ be an $\fM$-primary ideal. The only jumping numbers in the interval $(1,2]$ are the following:
 \begin{itemize}
  \item[$\bullet$] $\lambda+1$, where $\lambda\in(0,1]$ is a jumping number.
  \item[$\bullet$] $\lambda=\frac{k}{e_i}$, \hskip 2mm for $e_i< k\leqslant 2e_i$ { with $E_i$ dicritical divisor.}
\end{itemize}
\end{theorem}

\begin{proof}
Assume that a jumping number $\lambda \in (1,2]$ is not of the
announced types and consider its associated maximal jumping divisor
 $H_\lambda$. If $\lambda$ is not of the first type then
$m(\lambda)-m(\lambda-1)>0$. If it is not of the second type, then
$\rho_i=0$ for any $E_i\leqslant H_\lambda$. Both conditions cannot
be satisfied simultaneously by Proposition \ref{growth} so we get a
contradiction.
\end{proof}

%%%%%%%%%%%%%%%%%%%%%%%%%%%%%%%%%%%%%%%%%%%%%%%%%%%%%%%%%%%%%%%%%%%%

\begin{remark} Let  $\A \subseteq \cO_{X,O}$ be an $\m$-primary ideal. A generic element $f\in \A$ satisfies
 $\J(f^c)=\J(\A^c)$ for any $c \in (0,1)$ so Theorem \ref{jn1-2} says, roughly speaking, that the jumping numbers of $\A$ are
governed by the jumping numbers of a generic element $f\in \A$ and the dicritical divisors of $\A$.
\end{remark}

%%%%%%%%%%%%%%%%%%%%%%%%%%%%%%%%%%%%%%%%%%%%%%%%%%%%%%%%%%%%%%%%%

\section{Poincar\'e series of multiplier ideals}

Let $\A \subseteq \cO_{X,O}$ be an $\m$-primary ideal. In this section we will give a very simple description of the
Poincar\'e  series of multiplier ideals.
$$P_\A (t)= \sum_{c \in \bR_{> 0}} m(c) \hskip 1mm t^{c}= \sum_{c \in (0,1]} \sum_{k\in \bN} m(c+k) \hskip 1mm t^{c+k}$$
To such purpose we only need to control the following two issues: First we have to describe the multiplicities of the jumping numbers
in the interval $(0,1]$. This can be done using the formulas given in Theorem \ref{multiplicity1_c} or
Proposition \ref{virtual2}. Secondly, and equally important, we have to control the recurrence that these multiplicities satisfy.
As shown in Proposition \ref{growth}, dicritical components in the maximal jumping divisor allow us to describe the recurrence.

\vskip 2mm

The main result of this section is the fact that the Poincar\'e series of multiplier ideals is rational
in the sense that it belongs to the field of fractional functions $\bC(z)$, where the indeterminate $z$
corresponds to a fractional power $t^{1/e}$ for $e\in \bN_{>0}$ being the least common multiple of
the denominators of all jumping numbers. The formula for the Poincar\'e series that we obtain  is the following:

\begin{theorem}\label{serie_Poincare}
 Let $\A \subseteq \cO_{X,O}$ be an $\m$-primary ideal. The Poincar\'e series of $\A$ can be expressed as
$$ P_\A (t)= \sum_{c \in (0,1]}\left(\frac{m(c)}{1-t} + \rho_c  \frac{t}{(1-t)^2}\right) t^{c} $$
where $\rho_c=-F\cdot H_c$ and $H_c$ is the maximal jumping divisor associated to $c$.
\end{theorem}

\begin{proof}
Let $c \in (0, 1 ]$ be a real number. For any $k\in \bN$
we have, applying Proposition \ref{growth}
$$m(c + k)= m(c) + k \rho_c,$$ where $\rho_c= m(c+1)-m(c)=-F\cdot H_c .$ It follows
that
\begin{eqnarray*}
\sum_{k\geqslant 0}m(c + k)\hskip 1mm t^{c + k} & = & m(c ) \hskip 1mm t^{c} +
(m(c) +  \rho_c) \hskip 1mm t^{c +1} + (m(c) +  2\rho_c) \hskip 1mm  t^{c +2}  + \cdots   \\
&= &  \left(\frac{m(c)}{1-t}+ \rho_c \frac{t}{(1-t)^2}\right) t^{c}
\end{eqnarray*}
Thus we get the desired result.
\end{proof}

For the case of simple $\m$-primary ideals we can easily recover the extension
to the case where $X$ has rational singularities of  the main result of Galindo-Monserrat \cite{GM10}.
Our formulation slightly differs from theirs because we collect jumping numbers by the growth of the
multiplicities instead of its critical divisors.

\begin{corollary}\cite[Theorem 2.1]{GM10}
 Let $\A \subseteq \cO_{X,O}$ be a simple $\m$-primary ideal. The Poincar\'e series of $\A$ can be expressed as
$$ P_\A (t)=  \sum_{\substack{c \in (0, 1 ] \\ \rho_c=0}} \frac{m(c)}{1-t} t^{c} +
\sum_{\substack{c \in (0, 1 ] \\ \rho_c=1}}\left(\frac{m(c)}{1-t}+ \frac{t}{(1-t)^2}\right) t^{c} $$
\end{corollary}

\begin{proof}
Simple $\m$-primary ideals only have one dicritical divisor with excess $1$ so the result follows.
\end{proof}

\subsection{Hodge Spectrum}
Let  $X$ be a smooth complex variety of dimension $d$ and consider an hypersurface  with an isolated singularity at $O$ defined by
$f\in \cO_{X,O}$.
The Hodge spectrum $Sp(f)$ associated to $f$ was introduced by Steenbrink \cite{Ste77} using
the canonical mixed Hodge structure of the cohomology groups of
the Milnor fiber of $f$. It is a fractional polynomial
$$Sp(f)= \sum_{c\in [0,d]} n(c) \hskip 1mm t^{c},$$
where the rational number $c \in \bQ$ is an {\it exponent} or {\it spectral number} if its associated multiplicity $n(c)$ is
strictly positive. It is also known that the sum of all spectral numbers, counted with multiplicity, is equal to the
Milnor number of $f$ and that they are symmetric with respect to $\frac{d}{2}$, i.e. $n(c)=n(d-c)$

\vskip 2mm

Budur \cite{Bud03} established a nice relation between the Hodge spectrum  and the set of multiplier ideals.
More precisely, the multiplicity of  spectral numbers and the multiplicity of the so-called {\it inner} jumping
numbers coincide in the interval $(0,1]$. We point out that the usual jumping numbers
are inner jumping numbers whenever they are not integer numbers in the case of hypersurfaces with isolated singularities.

\vskip 2mm

In the case where $X$ has dimension two we can make a closer relationship between the Hodge spectrum
of a plane curve $f\in \cO_{X,O}$, that we assume as a generic  element of an $\fM$-primary ideal $\fa \subseteq \cO_{X,O}$,
and the Poincar\'e series of multiplier ideals of $\fa$. Roughly speaking,
the information given by the Hodge spectrum is equivalent, taking into account the symmetry with respect to $1$,
to the information given by the terms of the Poincar\'e series in the interval $(0,1)$. The aim of this section
is to strengthen this relationship recovering some old results on the Hodge spectrum of a plane curve by using our methods.

\vskip 2mm

The spectrum of a plane curve has been described by L{\^e} V{\u{a}}n Th{\`a}nh and Steenbrink in \cite{LVTS89}
(see also \cite{LVT88}, \cite{Sch90}).
For the convenience of the reader we will reformulate their result using the terminology we are considering in this paper.
To this aim, we consider a partial order on the exceptional components of the log-resolution. Since we are assuming that $O$ is a smooth point, the exceptional divisor is naturally a {\em rooted} tree of rational curves, where the root $E_1$ is the (strict transform of) the exceptional divisor of the blow-up of $O$. The partial order is then defined by the paths from $E_1$, i.e. $E_i$ precedes $E_j$ if $E_i$ belongs to the chain of components connecting $E_1$ and $E_j$. For any $i \neq 1$, we denote by $p\left(i\right)$ the index of the exceptional component immediately preceding $E_i$, so that $E_{p\left(i\right)}$ belongs to the chain connecting $E_1$ and $E_i$, and $E_i \cdot E_{p\left(i\right)}=1$.
%For any $E_i \neq E_1$, we will denote the immediately previous component as $E_{p(i)}$.
The set of rupture or dicritical divisors different from the root $E_1$ will be denoted $\mathcal{R}$, i.e.
$$\mathcal{R}=\{i\,|\,E_i \neq E_1\text{ is a rupture or dicritical divisor}\}.$$

\begin{theorem} \cite[Theorem 1.5]{LVTS89}
Let $f\in \cO_{X,O}$ be the equation of a plane curve with an isolated singularity at the origin $O$.
Let $c\in\Q$ be a rational number. Then, its associated multiplicity $n(c)$ in the Hodge
spectrum of $f$  is  $n(c)=n'(c)+n''(c)\,,$ where:
    \begin{itemize}
    \item[$\cdot$] $n'(c)=\#\left\{E_i\,|\,i\in\mathcal{R} \text{ and } E_i+E_{p(i)}\leqslant H_{c}\right\}$
    \item[$\cdot$] $n''(c)=\displaystyle \sum_{\substack{E_i\leqslant H_c\\i\in\mathcal{R} \cup\{1\}}}
    \left(-1+\sum_{E_j\in \adj(E_i)}\{c e_j\}+c\rho_i \right)$
   \end{itemize}

\end{theorem}

If we assume $f$ as a generic  element of an $\fM$-primary ideal $\fa \subseteq \cO_{X,0}$
we can recover this result  using the formula given in Theorem \ref{multiplicity1_c}.

 \vskip 2mm

\begin{proposition}
Let $f\in \cO_{X,O}$ be the equation of a plane curve with an isolated singularity at the origin $O$.
For any $c \in (0,1)$ we have $n(c) = m(c)$.

\end{proposition}

\begin{proof}
L{\^e} V{\u{a}}n Th{\`a}nh and Steenbrink's formula states that:
\begin{align*}
n(c) & = \#\left\{E_i\,|\,i\in\mathcal{R} \text{ and } E_i+E_{p(i)}\leqslant H_{c}\right\} + \displaystyle \sum_{\substack{E_i\leqslant H_c\\i\in\mathcal{R} \cup\{1\}}}\left(-1+\sum_{E_j\in \adj(E_i)}\{c e_j\}+c\rho_i \right) \\
& =  \#\left\{E_i\,|\,i\in\mathcal{R} \text{ and } E_i+E_{p(i)}\leqslant H_{c}\right\} - \#\left\{E_i\,|\,i\in\mathcal{R} \cup \left\{1\right\} \text{ and } E_i\leqslant H_{c}\right\}\\
& \qquad + \displaystyle \sum_{\substack{E_i\leqslant H_c\\i\in\mathcal{R} \cup\{1\}}}\left(\sum_{E_j\in \adj(E_i)}\{c e_j\}+c\rho_i \right) \\
& = - \#\left\{E_i\,|\,i\in\mathcal{R}, E_i \leqslant H_c \text{ and } E_{p(i)} \not\leqslant H_{c}\right\} - \delta + \displaystyle \sum_{\substack{E_i\leqslant H_c\\i\in\mathcal{R} \cup\{1\}}}\left(\sum_{E_j\in \adj(E_i)}\{c e_j\}+c\rho_i \right) \\
\end{align*}
where $\delta = 1$ if $E_1 \leqslant H_c$ and $\delta=0$ otherwise. Due to the rooted tree structure of the exceptional divisor, every connected component of $H_c$ has exactly one minimal component $E_i$ (the closest to $E_1$), and clearly $E_{p\left(i\right)} \not\leqslant H_c$ if $i\neq 1$. There is therefore a bijection between the set $\left\{E_i\,|\,i\in\mathcal{R}, E_i \leqslant H_c \text{ and } E_{p(i)} \not\leqslant H_{c}\right\}$ and the connected components of $H_c$ that contain some rupture or dicritical component but do not contain $E_1$. Hence we have proved
$$\#\left\{E_i\,|\,i\in\mathcal{R}, E_i \leqslant H_c \text{ and } E_{p(i)} \not\leqslant H_{c}\right\} + \delta = \#\left\{\substack{ \text{ connected components of } H_{c} \\ \text{ containing a divisor } E_i,
\hskip 2mm  i\in\mathcal{R} \cup\{1\}} \right\},$$
which gives the following expression for $n\left(c\right)$:

%{\it Claim:} We have
%$$\#\left\{E_i\,|\,i\in\mathcal{R} \text{ and } E_i+E_{p(i)}\leqslant
%H_{c}\right\} + \sum_{\substack{E_i\leqslant H_c\\i\in\mathcal{R} \cup\{1\}}} -1
%=  -\#\left\{\substack{ \text{ connected components of } H_{c} \\ \text{ containing a divisor } E_i,
%\hskip 2mm  i\in\mathcal{R} \cup\{1\}} \right\}
%$$

%First we notice that
%$$\#\left\{E_i\,|\,i\in\mathcal{R} \text{ and } E_i+E_{p(i)}\leqslant
%H_{c}\right\} + \sum_{\substack{E_i\leqslant H_c\\i\in\mathcal{R} \cup\{1\}}} -1
%=   \sum_{\substack{E_i\leqslant H_c\\E_{p(i)}\nless H_{c} \\ i\in\mathcal{R}\cup\{1\} }} -1.$$

%\vskip 2mm

%Due to the ordered tree structure of the dual graph and Theorem \ref{topo_H}
%we have that, for each $i\in\mathcal{R}$ such that $E_i\leqslant H_c$ and  $E_{p(i)} \nless H_c$, there exists a
%unique connected component $H \leqslant H_c$ such that  $E_i\leqslant H$. The proof of the claim just follows considering (if it exists)
%the connected component of $H_c$ that contains $E_1$.

%\vskip 2mm

%Therefore

\begin{equation} \label{eq-espectre-n}
n(c) = \displaystyle \sum_{\substack{E_i\leqslant H_c\\i\in\mathcal{R} \cup\{1\}}}\left(\sum_{E_j\in \adj(E_i)}\{c e_j\}+c\rho_i \right) -\#\left\{\substack{ \text{ connected components of } H_{c} \\ \text{ containing a divisor } E_i, \hskip 2mm  i\in\mathcal{R} \cup\{1\}} \right\}
\end{equation}

On the other hand, Corollary \ref{multiplicity2} gives (recall that $k_i \in \bZ$ because $O$ is a smooth point)
\begin{equation} \label{eq-espectre-m}
m\left(c\right) = \sum_{E_i \leqslant H_c} \left( \sum_{ E_j \in\adj(E_i)} \left\{c e_j\right\}+c\rho_i\right)-\#\left\{\text{connected components of }H_{c}\right\}.
\end{equation}
To prove that both formulas coincide, we have to consider the terms
\begin{equation*} %\label{eq-term}
\sum_{E_j\in \adj(E_i)}\{c e_j\}+c\rho_i
\end{equation*}
for the $E_i \leqslant H_c$ with $i \not \in \mathcal{R}\cup\{1\}$, as well as the connected components of $H_c$ containing only components of this kind.

\vskip 2mm

Consider first an $E_i$ which is not an isolated component of $H_c$. On the one hand, by Theorem \ref{thm-ends-Hc}, all its adjacent
components are contained in $H_c$, and hence $\sum_{E_j\in \adj(E_i)}\{c e_j\}=0$. Since it is not dicritical, $\rho_i=0$, and
therefore $E_i$ does not contribute to the first summand of $m\left(c\right)$. On the other hand, the connected component
$H$ of $H_c$ containing $E_i$ contains also either a rupture or dicritical component (again by Theorem \ref{thm-ends-Hc}), and hence
its contribution to the second summand of (\ref{eq-espectre-m}) is already taken into account in (\ref{eq-espectre-n}).

\vskip 2mm

To finish the proof, it remains to consider the $E_i$ which are isolated components of $H_c$. In this case, Theorem \ref{thm-ends-Hc}
says that the contribution of $E_i$ to the first term of (\ref{eq-espectre-m}) is $\sum_{ E_j \in \adj(E_i)} \{c e_j\}=1$, which cancels
with the contribution to the number of connected components.
\end{proof}

\end{document}